\documentclass[a4paper,11pt]{amsart}
\usepackage[hyperindex,breaklinks]{hyperref}
\usepackage[linesnumbered,ruled,vlined]{algorithm2e}

\usepackage[usenames,dvipsnames,svgnames,table]{xcolor}
\usepackage[all]{xy}
\usepackage{enumitem}
\newcommand{\K}{\mathbb{K}}
\newcommand{\F}{\mathbb{F}}

\usepackage[active]{srcltx}
\usepackage{amsmath}
\usepackage{amssymb}
\usepackage{amscd}
\usepackage{amsthm}
\usepackage[latin1]{inputenc}
\usepackage{mathrsfs}  

\usepackage{nicefrac}

\newtheorem{teo}[equation]{Theorem}
\newtheorem{problem}[equation]{Problem}
\newtheorem{definition}[equation]{Definition}
\newtheorem{prop}[equation]{Proposition}

\theoremstyle{definition}

\newtheorem{ex}[equation]{Example}
\newtheoremstyle{dico}
 {\baselineskip}   
  {\topsep}   
  {}  
  {0pt}       
  {} 
  {.}         
  {5pt plus 1pt minus 1pt} 
  {}          
\theoremstyle{dico}
\newtheorem{say}[equation]{}
\numberwithin{equation}{section}

\newcommand{\ra}{\rightarrow}

\newcommand{\Zeta}{{\mathbb{Z}}}
\newcommand{\N}{{\mathbb{N}}}

\newcommand{\meno}{^{-1}}

\newcommand{\alfa}{\alpha}
\newcommand{\alf}{\alpha}

\newcommand{\Aut}{\operatorname{Aut}}

\newcommand{\ord}{{\operatorname{ord}}}

\renewcommand{\setminus}{-}

\renewcommand{\phi}{\varphi}
\newcommand{\lds}{\dotsc}
\newcommand{\cds}{\cdots}
\newcommand{\cd}{\cdot}

\newcommand{\sx}{\langle}
\newcommand{\xs}{\rangle}

\newcommand{\lra}{\longrightarrow}

\newcommand{\ga}{\gamma}
\newcommand{\Ga}{\Gamma}

\newcommand{\id}{\operatorname{id}}

\newcommand{\PP}{\mathbb{P}}   
   
\renewcommand{\phi}             {\varphi}

\newcommand{\HH}{\mathbf{H}}

\newcommand{\Datum}{\theta}

\newcommand{\Z}{\mathbb{Z}}
\newcommand{\zio}{\zopfr\times \Aut G}
\newcommand{\zopf}{\operatorname{B}}
\newcommand{\zopfr}{\operatorname{B}_r}

\newcommand{\quozient}[2]       {\raisebox{.3EM}{$#1$} \hspace {-.15ex} 
                                  \Big / \hspace{-.15ex}
                                  \raisebox{-.4EM}{$#2$}}

\newcommand{\mm}{{\mathbf{m}}}
\newcommand{\HM}{\HH_\mm}

\newcommand{\mihi}[1]{}

\newcommand{\magma}{{\texttt{MAGMA}}\xspace}

\newcommand{\auta}{\Aut^* \Ga_r}
\newcommand{\aug}{\Aut G}
\newcommand{\Out}{\operatorname{Out}}
\newcommand{\outa}{\Out^*\Ga_r}
\newcommand{\braid}{\zopfr} 

\newcommand{\data}{\mathscr{D}}

\newcommand{\datar}{\mathscr{D}^r}

\newcommand{\datarg}{\mathscr{D}^r_g}

\newcommand{\codim}{\operatorname{codim}}
\newcommand{\ts}{\tilde{\sigma}}

\newcommand{\xgm}{\data_{G,\mm}}
\newcommand{\pro}{\rho}

\newcommand{\siggr}{\mathfrak{S}_{d,g_{\max}}}

\newcommand{\abs}[1]{\left\vert#1\right\vert}

\newcommand{\info}{{\HH}}
\newcommand{\inn}{\operatorname{inn}}
\newcommand{\Inn}{\operatorname{Inn}}

\usepackage{graphicx}

\begin{document}

\author{Diego Conti, Alessandro Ghigi, Roberto Pignatelli} \title
[Topological types] {Topological types of actions
  on curves}

\address{Universit\`{a} di Pisa}

\email{diego.conti@unipi.it}

\address{Universit\`{a} di Pavia}

\email{alessandro.ghigi@unipv.it}

\address{Universit\`{a} di Trento}

\email{Roberto.Pignatelli@unitn.it}

\subjclass[2020]{Primary:   
14H37, 
  Secondary:
 14Q05, 
  57M60, 
   14H15, 
  14J10. 
}

\thanks{ The authors were partially supported by INdAM (GNSAGA).  The
  second author was partially supported also by MIUR PRIN 2017
  ``Moduli and Lie Theory'' , by FAR 2016 (Pavia) ``Variet\`a
  algebriche, calcolo algebrico, grafi orientati e topologici'', by
  MIUR, Programma Dipartimenti di Eccellenza (2018-2022) -
  Dipartimento di Matematica ``F. Casorati'', Universit\`a degli Studi
  di Pavia. The third author was partially supported also by MIUR PRIN
  2015 ``Geometry of Algebraic Varieties'' and by MIUR PRIN 2017
  ``Moduli Theory and Birational Classification''.}

\begin{abstract}
  We describe an algorithm that constructs a list of all topological
  types of holomorphic actions of a finite group on a compact Riemann
  surface $C$ of genus $g \geq 2$ with $C/G \cong {\mathbb P}^1$.
\end{abstract}

\maketitle



\section{Introduction}

Galois covers of the projective line often give interesting examples
of algebraic curves of genus $g\geq 2$.  Any such cover is a compact
Riemann surface $C$, endowed with an action of a finite group $G$ such
that $C /G \cong \PP^1$. Studying compact Riemann surfaces with a
$G$-action subject to the latter condition and Galois covers of the
line is equivalent.

To any $G$-action on a Riemann surface one can attach a topological
invariant, the \emph{topological type} (see Section \ref{sec:fam} for
the precise definitions).  It turns out that for every fixed
topological type there is a sort of universal family containing all
the covers with the given topological type. These families are often
very interesting and represent interesting loci in the moduli space of
curves.

Moreover it follows from the existence and the properties of these
families that two Riemann surfaces with an action of the same group
are deformation equivalent (through Riemann surfaces with an action)
if and only if they have the same topological type (see \S \ref{fami} for
more details on this point).  Thus, if one starts from a collection of
Riemann surfaces and constructs new algebraic varieties out of them,
the topological type controls the deformation equivalence class of the
varieties one obtains.

For these reasons it is very useful to have a list of all the
possibile topological types at least for reasonably small genus. This
is exactly the problem we address: for a positive integer $g \geq 2$,
describe explicitly the set of topological types of (faithful) actions
of a finite group on a Riemann surface of genus $g$ with quotient
isomorphic to ${\mathbb P}^1$. To our knowledge a complete answer is
known only for very small values of $g$. The papers
\cite{kurikuri,kuribayashi-akikazu-kimura} give a complete
classification of the topological types of group actions on Riemann
surfaces of genus $g \leq 5$ without any assumption on the quotient
$C/G$.  For higher genus, we do not know of any classification result
of this type. An important contribution in this direction is the
systematic work of Paulhus \cite{paulhus}. She introduces a very
efficient algorithm that yields a database of pairs $(C,G)$ where $C$
is a Riemann surface of genus among $2$ and $15$ and $G$ is a finite
group acting on it with quotient $C/G$ of genus zero. Her database
contains representatives for all possible topological types of groups
actions in this range, but does not give a complete answer to the
question (up to genus $15$) since it does not indicate when two data
in the database share the same topological type. For instance,
Paulhus' database contains $174121$ data for $g=15$, which according
to our computations correspond to $768$ topological types.  (Added in
proof: another important work that we discovered at a very late stage,
is \cite{AAA}. We mention also the paper \cite{BBB} where the
topological classification is useful in decomposing Jacobians with
group action.)


Fix $g\geq 2$, a finite group $G$ and $r\geq 3$.  It is well known
that
the topological types of holomorphic $G$-actions on a genus $g$
surface with $r$ branch points correspond bijectively to the points of
the quotient $\datarg(G) / \info$ where $\datarg(G)$ is a finite (but
possibly huge) set and $\info$ is an infinite group acting on it (this
fundamental fact is recalled with some details in Section
\ref{sec:fam}; see also \eqref{def-HH} for the definition of
$\info$). Therefore, the object of this paper is the description of
the quotient $\datarg(G)/\info$ for fixed $g$ and $G$.  Algorithms for
this kind of computation are already known and have been used in
several papers, like for example
\cite{BCGP,fgp,ContiGhigiPignatelli:SomeEvidence,gleissner}, just to
quote the ones closer to our approach.  As the genus $g$ increases,
the set $\datarg(G)$ becomes quite large, and finding an economical
way of performing the computation becomes essential. In the present
paper we describe an efficient algorithm that computes
$\datarg(G)/\info$ for given $g$ and $G$. The computation of
$\datarg(G)$ is based on our work
\cite{ContiGhigiPignatelli:SomeEvidence}, and uses some of the ideas
of \cite{paulhus}; the identification of the quotient is new.

We implemented the algorithm using \magma\cite{MA}; our implementation
is available at \cite{gullinbursti}. Running the code over several
months on a computer with 56 Intel Xeon 2.60GHz CPU and 128 GB of RAM
we have been able to compute $\data_g(G)/\info$ for all groups $G$ and
$g\leq 39$, with the exception of only three cases.  These exceptions
share the same group, $G= (\Z_3\times\Z_3)\rtimes\Z _2$,
and are respectively in genus $g=28,34,37$.  See Table
\ref{table:numberofssgbyg} for an account of how many topological
types exist for each genus. See \S \ref{say:results} for some
perspectives on future work related to the group $G$ above and similar
ones.

The results of our computations are collected in a database, available at
\begin{quote}
  \href{https://mate.unipv.it/ghigi/tipitopo}{\url{https://mate.unipv.it/ghigi/tipitopo}}.
\end{quote}
The database also contains several other data classifying actions on surfaces of
higher genus (up to $g=100$) under some further constraints. We refer
to the website for the exact indication of these constraints.


The paper is organized as follows: in \S \ref{sec:fam} we recall the
theoretical background, introducing the precise definitions of
$\datarg(G)$, $\info$; in \S \ref {sec:algoritmo} we describe the
algorithm, and give some details on its implementation.

\smallskip {\bfseries \noindent{Acknowledgements.} }
The authors wish to thank Matteo Penegini, Fabio Perroni and Michael
Loenne for interesting discussions, and Fabio Gennai for crucial
technical help in the preparation of the website.

\section{Topological types and families of
  \texorpdfstring{$G$}{G}-curves}
\label{sec:fam}

The goal of this section is to recall without proofs the mathematics
behind our computation.

A $G$-\emph{curve} is a smooth projective curve $C$ over the complex
numbers together with an effective algebraic action of the group $G$.
We always assume that the genus of $C$ is at least $ 2$ and hence that
$G$ is finite.  We also assume throughout the paper that the quotient
$C/G$ is isomorphic to the projective line $ \PP^1$, i.e. the
projection $C \ra C/G$ is a Galois cover of $\PP^1$.  Under this
assumption it is completely equivalent to study $G$-curves or Galois
covers of the line.

\begin{definition}
  If $C$ and $C'$ are two $G$-curves we say that they are
  \emph{topologically equivalent} or that they have the same
  \emph{topological type} if there exist $\eta \in \Aut G$ and an
  orientation preserving homeomorphism $f: C \ra C'$ such that
  $f(g\cd x) = \eta (g) \cd f(x)$ for $x\in C$ and $g\in G$. We say
  that $C$ and $C'$ are $G$-\emph{isomorphic} if moreover $f$ is a
  biholomorphism.
\end{definition}
These concepts are sometimes called \emph{unmarked} topological type
and isomorphisms, but we will drop the `unmarked' since we do not need
to consider their marked counterparts.

For $r\geq 3$ let $\Ga_r $ denote the group
\begin{gather*}
  \Ga_r = \sx \ga_1, \lds, \ga_r\, | \, \ga_1 \cds \ga_r\xs.
\end{gather*}

\begin{definition}
  \label{def:data}
  If $G$ is a finite group an $r$-datum is an epimorphism
  $\theta : \Ga_r \ra G$ is such that $\theta(\ga_i)\neq 1$ for
  $i=1, \lds, r$.
\end{definition}
The \emph{signature} of $\theta$, denoted $\mm(\theta)$ or simply
$\mm$, is the vector
\begin{gather*}
  \mm=(m_1,\dotsc, m_r)
\end{gather*}
where $m_i : = \ord\, \theta(\ga_i)$.  The \emph{genus} of $\theta$,
denoted by $g(\theta)$, is defined by the Riemann-Hurwitz formula:
\begin{gather}
  \label{RH}
  2(g(\theta)-1)=|G|\left(-2+\sum_{i=1}^{r}\left(1-\frac{1}{m_{i}}\right)\right)
\end{gather}
By covering theory if one chooses $r$ distinct points $x_1, \lds, x_r$
in $\PP^1$, a base point $x_0 \in \PP^1 \setminus \{x_1, \lds, x_r\}$
and an isomorphism among $\Ga_r$ and
$\pi_1(\PP^1 \setminus \{x_1, \lds, x_r\}, x_0)$, then a datum
corresponds to a topological Galois covering of
$C^* \ra \PP^1 \setminus \{x_1, \lds, x_r\} $ with structure group
$G$. By Riemann's Existence Theorem (see e.g.  \cite[ch. III \S
4]{M95} or \cite{debes}) this compactifies to a $G$-covering
$C \ra \PP^1$, where the genus of $C$ is $g(\theta)$.

We let $\datarg(G)$ denote the set of all $r$-data of genus $g$
associated with the group $G$.

\begin{definition} \label{sayautstar} Denote by
  $\auta \subset \Aut \Ga_r$ the subgroup of automorphisms $\nu$
  satisfying:
  \begin{enumerate}
  \item for $i=1, \lds, n$ the element $\nu (\ga_i)$ is conjugate to
    $\ga_j$ for some $j$;
  \item the automorphism of $H^2(\Ga_r,\Zeta)$ induced by $\nu$ is the
    identity.
  \end{enumerate}
\end{definition}
The second condition means that $\nu$ is orientation-preserving: if we
identify $\Ga_r $ with $ \pi_1(\PP^1-\{x_1, \lds, x_r\})$
appropriately (using a so-called geometric basis) then $\nu$ is
represented by an orientation preserving self-homeomorphism of
$\PP^1 \setminus \{x_1, \lds, x_r\}$ (Dehn-Nielsen-Baer theorem, see
e.g. \cite{fmarga,ivanov,zieschang}).

\begin{say}
  \label{auta-azione}
  The group $\Aut^*\Gamma_r \times \aug $ acts on the set
  $\datarg(G) $ by the rule
  \begin{gather*}
    (\nu, \eta) \cd \theta : = \eta \circ \theta \circ \nu\meno,
  \end{gather*}
  where $(\nu, \eta) \in \Aut^*\Gamma_r\times \Aut G$ and
  $\theta\in \datarg(G)$ is a datum.  Moreover
  $\Inn \Ga_r \subset \Aut^* \Ga_r$ and we set
  \begin{gather*}
    \outa: = \frac{\auta}{\operatorname{Inn} \Ga_r}.
  \end{gather*}
  This group has a presentation with generators
  $\sigma_1, \lds, \sigma_{r-1} $ and relations
  \begin{gather}
    \label{umpf1}
    \sigma_i \sigma_j = \sigma_j \sigma_i \quad \text{for} \, \,
    |i-j|\geq 2, \qquad
    \sigma_{i+1}\sigma_i\sigma_{i+1}=\sigma_i\sigma_{i+1}\sigma_i, \\
    \label{umpf2} \sigma_1 \cds \sigma_{r-2} \sigma_{r-1}^2
    \sigma_{r-2} \cds \sigma_1 = 1,\qquad ( \sigma_1 \cds \sigma_{r-1}
    )^r =1.
  \end{gather}
  Instead the braid group $\zopf_r$ is the group generated by
  $\sigma_1, \lds, \sigma_{r-1}$ subject only to relations
  \eqref{umpf1}.  Define $\ts_i: \Ga_r \ra \Ga_r$ as follows:
  \begin{gather}
    \label{deftztz}
    \begin{gathered}
      \ts_i (\gamma_i) = \gamma_{i+1}, \quad \ts_i
      (\gamma_{i+1}) = \gamma_{i+1} ^{-1} \gamma_i \gamma_{i+1}, \\
      \ts_i (\gamma_j ) = \gamma_j \quad \text{for }j \neq i, i+1.
    \end{gathered}
  \end{gather}
  Then $\ts_i$ is an automorphism of $\Ga_r$ and it is called the
  \emph{$i$-th Hurwitz move}.  The Hurwitz moves $ \ts_1,$ $ \lds,$
  $ \ts_{r-1}\in\auta$ satisfy the relations \eqref{umpf1}. Thus there
  is a (unique) morphism $\phi : \zopf_r \ra \auta$ such that
  $\phi(\sigma_i) : = \ts_i$.  Using $\phi$ we let
  $\zopf_r \times \Aut G $ act on $\datarg$: if
  $(\eta,\sigma) \in\zopfr \times \Aut G$ and $\theta \in \datarg$,
  then
  \begin{gather*}
    (\eta,\sigma) \cd \theta := \eta \circ \theta \circ
    \phi(\sigma)\meno.
  \end{gather*}
  The composition
  \begin{gather*}
    \zopfr \stackrel{\phi}{\lra} \auta \lra \outa
  \end{gather*}
  maps $\tilde{\sigma}_i $ to $\sigma_i$, so is surjective. It follows
  that $\phi( \zopfr)\cd \operatorname{Inn} \Ga_r = \auta$.  For
  $a\in \Ga_r$ let $\inn_a : \Ga_r \ra \Ga_r$ be conjugation by $a$:
  $\inn_a(x) := axa\meno$. Then for $\theta \in \datarg$, we have
  \begin{gather*}
    \theta\circ \inn _a = \inn_{\theta(a)} \circ \theta.
  \end{gather*}
  Hence the actions of $ \auta \times\Aut G $ and of
  $\zopfr\times\Aut G $ on $\datarg$ have the same orbits and
  \begin{gather*}
    \datarg / ( \auta \times \Aut G ) = \datarg / ( \zopfr \times\Aut
    G).
  \end{gather*}
  The orbits of the $ \zopfr \times\aug $--action are called
  \emph{Hurwitz equivalence classes} and elements in the same orbit
  are said to be \emph{Hurwitz equivalent}.

\end{say}

\begin{teo}
  Fix $g \geq 2 $ and a finite group $G$, the topological types of
  $G$-curves $C$ with $g(C)=g$, $g(C/G)= 0$ and $r$ branch points are
  in bijection with the set
  \begin{gather}
    \label{set}
    \datarg(G) / \left( \zopfr \times\Aut G \right).
  \end{gather}
\end{teo}

A proof can be found for example in \cite[Section 5]{gt}.

\begin{say}\label{fami}
  It often happens that one is not interested in a precise $G$-curve,
  but rather in the whole family of $G$-curves of a given topological
  type.  Indeed, there is a sort of universal family containing all
  $G$-curves of a given topological type.  These ``universal''
  families have been widely studied and used in the literature for
  several purposes in the last decades, see
  e.g. \cite{BCGP,baffo-linceo,clp,clp2,moonen-special,fgp,fn,friedv,penegini2013surfaces,perroni,volk}.
  They were first constructed by Gonz\'alez-D\'{i}ez and Harvey
  \cite{gabino} using Teichm\"uller theory.  There are other ways to
  construct it.  Recently in \cite{gt} the second author and Tamborini
  gave a different construction of these families and corrected an
  inaccuracy in \cite{gabino}.  The precise statement is a bit long
  and there is no need to recall it here in full detail. At any rate,
  the end result is that for any topological type there is a family of
  $G$-curves whose fibres all have the given topological type, and
  which is universal in the following sense: every $G$-curve with the
  given topological type appears as a fibre of the family and it
  appears at most a finite number of times.  Moreover the base of this
  family is an \'etale cover of the $\mathsf{M}_{0,r}$, so in
  particular it is smooth and connected.  This ``universal'' family is
  not unique, but only unique up to the equivalence relation generate
  by finite \'etale pull-backs.  We refer to the Introduction in
  \cite{gt} for full details.

  An important consequence of this theorem is that topological types
  and deformation equivalence coincide for $G$-curves. Indeed if $C$
  and $C'$ are $G$-curves with the same topological type, then they
  both appear as fibres of a common universal family and therefore
  they are deformation equivalent (through $G$-curves). The converse
  is obvious.  This fact is very important in the applications to the
  construction of new deformation types of algebraic varieties as in
  \cite{BCGP,BP12,penegini2011surfaces,penegini2013surfaces,polizzi,
    carnovale polizzi,mistretta polizzi,FP,gleissner,BP16,polizzi
    pignatelli, pignatelli}.  It follows from this discussion that the
  classification of \emph{universal families} and of deformation
  equivalence classes of $G$-curves are both equivalent to the
  classification of topological types.  Hence again these problems
  boil down to studying the quotient in \eqref{set}.  This is a strong
  additional motivation --- in fact, our original motivation --- for
  the classification of topological types.
\end{say}

\section{The algorithm}
\label{sec:algoritmo}

We illustrate an algorithm to attack the following:
\begin{problem}
  \label{problem:classification}
  Given a number $g_{\max}$ list all the topological types of
  $G$-curves with $g(C) \leq g_{\max}$ and $C/G \cong \PP^1$.
\end{problem}
After fixing the genus $g$, the group $G$ and the number of branch
points $r$, this amounts to listing representatives of the quotient in
\eqref{set}.

\begin{say}
  \label{ignore_order}
  We use a refinement of the algorithm illustrated in
  \cite{ContiGhigiPignatelli:SomeEvidence}, which lists $r$-data
  forming counter-examples to the Coleman-Oort conjecture up to the
  action of $\Aut G$. This algorithm uses signature as an invariant
  for the classification, as was done in \cite{BCGP,fgp}.  A
  \emph{spherical system of generators} of the group $G$ is a list
  $(g_1, \lds, g_r)$ such that (1) $g_i\neq 1$ for any $i$, (2) $G$ is
  generated by $g_1, \lds, g_r$ and (3) $g_1\cds g_r = 1$.  Having
  fixed a finite group $G$, giving an $r$-datum $\theta$ is equivalent
  to giving a spherical systems of generators $(g_1, \lds, g_r)$ of
  $G$: simply define $\theta$ or $g_i$ by the relation
  $\theta(\ga_i)=g_i$.  From now on we will identify data and
  spherical systems of generators and we will write a datum in
  $\datarg(G)$ as $(g_1, \lds, g_r)$.  Signature defines a map
  \begin{equation}
    \label{eqn:signaturemap} 
    \datarg(G) \ra \N^r, \quad (g_1, \dots, g_r)\mapsto ( \ord(g_1), \lds, \ord(g_r)).
  \end{equation}
  With the interpretation of $\datarg(G)$ just described, we have
  $\datarg (G) \subset G^r$, and the action of $\zio$ on $\datarg(G)$
  described in \ref{auta-azione}
  extends to an action
  on $G^r$: $\Aut G$ acts componentwise,
  \begin{gather*}
    \eta\cdot (g_1,\dotsc, g_r)=(\eta(g_1),\dotsc, \eta(g_r),
  \end{gather*}
  while the generator $\sigma_ i $ of $\zopfr$ acts as follows
  \begin{gather}
    \label{azione}
    \sigma_i \cd (g_1, \lds, g_r) = (g_1,\dots,\underbrace{g_i
      g_{i+1}g_i^{-1}}_i,\underbrace{g_{i}}_{i+1},\dotsc, g_r).
  \end{gather}
  Denote by $\Sigma_r$ the symmetric group.  There is a surjective
  morphism
  \begin{equation}
    \label{rhouta}
    \pro : \zopfr \ra \Sigma_r,  
  \end{equation}
  which maps the $i$-th Hurwitz move $\sigma_i$ to the transposition
  $(i, i+1)$.  The map \eqref{eqn:signaturemap} is $\rho$-equivariant:
  if $\psi$ is in $\zopfr$ and $\sigma=\rho (\psi)$ then
  $( \psi \cd (g_1 , \lds, g_r) )$ is mapped to
  $(\ord(g_{\sigma_1}) , \lds, \ord(g_{\sigma_r}))$.
  
  Recall that Hurwitz equivalence in $\datarg(G)$ is defined in terms
  of the action of $\braid\times\Aut G$; thus,
  \eqref{eqn:signaturemap} maps Hurwitz equivalence classes onto
  $\Sigma_r$-orbits in $\Zeta^r$. This shows that every Hurwitz
  equivalence class has a representative with signature of the form
  \[\mm=(m_1,\dotsc, m_r), \quad m_1\leq \dots \leq m_r,\]
  and $\mm$ is uniquely determined by the equivalence class.
\end{say}

\begin{say}
  As a first step we wish to determine the set of all possible
  signatures.  We iterate over the order $d=|G|$. For fixed $d$, let
  $\mathfrak{S}_{d,g_{\max}}$ be the set of finite sequences
  $\mm=(m_1,\dotsc, m_r)$ such that
  \begin{enumerate}[label=(S\arabic*)]
  \item \label{item:S1} $3\leq r\leq \frac{4(g_{\max}-1)}d+4$;
  \item \label{item:S2} each $m_i$ is a divisor of $d$;
  \item \label{item:S3} $2\leq m_i\leq d$;
  \item \label{item:S4} $g$, determined by $d$ and $\mm$ as in
    \eqref{RH}, is an integer between $2$ and $g_{\max}$.
  \item \label{item:S4b} if $r=3$, $d$ satisfies the bound of
    \cite[Appendix 1]{Conder:Large group actions on surfaces} as long
    as $g \leq 301$.
  \item \label{item:S5}$m_1\leq \dots \leq m_r$;
  \end{enumerate}
  It follows from the same arguments as in the proof of Hurwitz
  theorem that the signature of any $G$-curve satisfies \ref{item:S1},
  see e.g.  Lemma 3.2 (b) in \cite
  {ContiGhigiPignatelli:SomeEvidence}.  Notice that we do not exclude
  $r=3$ or cyclic groups here, unlike in
  \cite{ContiGhigiPignatelli:SomeEvidence}.
  
  Computing $\mathfrak{S}_{d,g_{\max}}$ is not difficult (see
  e.g. \cite[Algorithm 1]{ContiGhigiPignatelli:SomeEvidence}). In our
  implementation, we found it convenient to store the resulting
  signatures on disk for later retrieval; this allows us to iterate
  through signatures with fixed genus $g$ at a later stage, in order
  to compute the sets $\mathscr{D}_g(G)$. Upon retrieving signatures
  with fixed genus, we make use of the inequalities $d\leq 84(g-1)$
  (for $r=3$) and $d\leq 12(g-1)$ (for $r\geq 4$, see \cite[Lemma
  3.2.(c)]{ContiGhigiPignatelli:SomeEvidence}), which do not appear in
  the definition of $\mathfrak{S}_{d,g_{\max}}$ since they are
  consequences of \eqref{RH}.
\end{say}

\begin{say}
  Problem~\ref{problem:classification} can be addressed by iterating
  through the signatures in $\mathfrak{S}_{d,g_{\max}}$. For each
  signature, we iterate through isomorphism classes of groups $G$ of
  order $d$. Some groups can be eliminated right away, namely:
  \begin{itemize}
  \item groups $G$ that do not contain elements of order $m_i$ for
    some $m_i$ in the signature;
  \item groups $G$ that contain elements of order greater than $4g+2$
    (for $r=3$) or $4(g-1)$ (for $r>3$);
  \item groups that cannot be generated by $r-1$ elements because
    their abelianization cannot.
  \end{itemize}
  Similar exclusions are listed in \cite[Algorithm
  2]{ContiGhigiPignatelli:SomeEvidence}, with some differences due to
  the fact that we allow $r=3$ and cyclic groups here.
\end{say}

\begin{say}
  What remains to be done is, in fact, the most complicated and most
  novel part of our work, namely producing an algorithm to classify
  spherical systems of generators for fixed $G$ and $\mm$. The
  construction of this algorithm will take the rest of this section.

  Fix a group $G$ of order $d$ and a signature $\mm$, and let $\xgm$
  be the set of spherical systems of generators of $G$ with signature
  $\mm$. The group
  \[(\braid)_{\mathbf{m}}=\{g\in \braid\mid \rho(g)\cdot\mm=\mm\}\]
  acts on $\xgm$, and so does $\Aut G$. We will need the following:
\end{say}
\begin{prop}
  \label{prop:classify}
  Given a group $G$ of order $d$ and $\mm$ in $\siggr$, two elements
  $\Datum,\Datum'$ of $\xgm$ are Hurwitz equivalent if and only if
  they are in the same $(\zopfr)_{\mm}\times \aug$-orbit.
\end{prop}
\begin{proof}
  If $\Datum$ and $\Datum'$ are in the same
  $(\zopfr)_{\mm}\times \aug$-orbit, they are obviously Hurwitz
  equivalent.

  Conversely, suppose $\Datum'=(\nu,\alpha)\cdot\Datum$, with
  $\nu\in\braid$, $\alpha\in\Aut G$; by equivariance of
  \eqref{eqn:signaturemap}, the signature of $\Datum'$ is
  $\mm=\rho(\nu)\cdot\mm$, so $\nu$ lies in $(\braid)_{\mathbf{m}}$.
\end{proof}

\begin{say}
  \label{say:determinesection}
  Set for simplicity
  \begin{gather}
    \label{def-HH}
    \HH:= \zopfr \times \Aut G.
  \end{gather}
  Then $ \HH_\mm = (\zopfr)_\mm \times \Aut G $. We are reduced to the
  following problem: given a group $G$ of order $d$ and $\mm$ in
  $\siggr$, determine a \emph{section} for the action of $\HM$ on
  $\xgm$, that is to say a subset of $\xgm$ that contains exactly one
  element in each $\HM$-orbit.

  The approach used in \cite{BCGP,fgp} was to iterate through lists of
  elements in
  \begin{gather*}
    G^{\mathbf{m}}=\{(g_1,\dotsc, g_r)\in G^r\mid
    \ord(g_1)=m_1,\dotsc, \ord(g_r)=m_r\}
  \end{gather*}
  and identify those for which $\prod g_i=1$ and
  $\langle g_1,\dotsc, g_r\rangle =G$. This produces a set of
  spherical systems of generators which can become quite large as $r$
  or $|G|$ increase; the fact that we are ultimately interested in
  extracting a representative for each orbit of $\HM$ suggests that an
  alternative approach could be more suited to our goal.

  A preliminary observation is that, whilst the group $\HM$ is
  infinite, it can be replaced by its image in
  $\operatorname{Perm}(\xgm)$, the symmetric group over $\xgm$, which
  is finite. In prior algorithms, precisely first in the paper
  \cite{BCGP} and then in
  \cite{BP12,fgp,fpp,FP,BP16,gleissner,ContiGhigiPignatelli:SomeEvidence},
  the orbit of an element is calculated by a heavy recursive procedure
  that builds an increasing chain of sets by the action of a fixed set
  of generators of $\HH$ and stops when the chain stabilizes. By
  contrast, here we build first the image of $\HM$ as subgroup of
  $\operatorname{Perm}(\xgm)$, then compute directly all orbits
  without any recursion.

  After this first step, the problem boils down to extracting a
  section for the action of a finite group on a finite set. This can
  be achieved efficiently in \magma using \texttt{GSets}. It is then
  the size of the finite group and of the finite set $\xgm$ that
  determine memory use and execution time.

  The basic idea to reduce both is to split the action of $\HM$ on
  $G^\mm$ into a large number of actions, each one by a much smaller
  group acting on a much smaller set.  The key observation is the
  following.  Consider the set $\mathcal{C}_G$ of conjugacy classes in
  $G$, and consider the commutative diagram
  \begin{equation}
    \label{diag:Gr}
    \xymatrix{ G^r\ar[r]^p\ar[dr]^\pi & (\mathcal{C}_G)^r\ar[d]\\
      &(\mathcal{C}_G)^r/\Sigma_r}
  \end{equation}
  where $p(g_1,\dotsc, g_r)=([g_1],\dotsc, [g_r])$.  Setting
  $\mathcal{C}_G^{\mathbf{m}}:= p (G^\mm)$, we obtain a commutative
  diagram
  \begin{equation}
    \label{diag:Gm}
    \xymatrix{G^{\mathbf{m}}\ar[r]^p\ar[dr]^\pi & \mathcal{C}_G^{\mathbf{m}}\ar[d]\\
      &\mathcal{C}_G^{\mathbf{m}}/{(\Sigma_r)_\mathbf{m}}}
  \end{equation}
  with $(\Sigma_r)_\mathbf{m}$ denoting the stabiliser of $\mathbf{m}$
  in $\Sigma_r$.  The group $\HH$ acts also on $(\mathcal{C}_G)^r$ in
  the following way:
  \begin{equation*}
    \begin{aligned}
      \eta\cdot  ([g_1],\dotsc, [g_r])&=(\eta([g_1]),\dotsc, \eta([g_r]), \qquad \eta \in \Aut G\\
      \sigma_i\cdot ([g_1],\dotsc, [g_r])&=
      ([g_1],\dots,\underbrace{[g_{i+1}]}_i,\underbrace{[g_{i}]}_{i+1},\dotsc,
      [g_r]).
    \end{aligned}
  \end{equation*}
  This means that if $(\eta,\psi) \in \HH$ and $ \sigma = \rho(\nu)$,
  then
  \begin{gather}
    \label{pVppa}
    (\eta,\nu)\cd ([g_1],\dotsc, [g_r]) = ([\eta(g_{\sigma_1})], \lds,
    [\eta(g_{\sigma_r})]).
  \end{gather}
  For $\tilde{c} \in (\mathcal{C}_G)^r$, let $\HH_{\tilde {c}}$
  denotes the stabilizer for this action.

  Since $(\mathcal{C}_G)^r/\Sigma_r$ is the quotient of
  $(\mathcal{C}_G)^r$ by $\braid$, it has an induced action of
  $\HH$ 
  with the factor $\braid$ acting trivially.  In other words $\Aut
  G$ acts naturally on $(\mathcal{C}_G)^r/\Sigma_r$ and we let
  $\zopfr$ act trivially on this set. With this understood the map
  $p$ and the whole Diagram \eqref{diag:Gr} is
  $\HH$-equivariant, while \eqref{diag:Gm} is $\HM$-equivariant.
\end{say}

Notice that a section $S$ for the action of $\Aut
G$ on
$\mathcal{C}_G^{\mathbf{m}}/{(\Sigma_r)_\mathbf{m}}$ can be
constructed adapting \cite[Algorithm
4]{ContiGhigiPignatelli:SomeEvidence}.  From this, we recover a
section for the action of $\HM$ on $G^{\mathbf{m}}$ as follows:
\begin{prop}
  \label{prop:section}
  Let $S\subset
  \mathcal{C}_G^{\mathbf{m}}/{(\Sigma_r)_\mathbf{m}}$ be a section for
  the action of $\Aut G$, and let $\tilde
  S$ be a subset of
  $\mathcal{C}_G^{\mathbf{m}}$ that projects one-to-one onto
  $S$. Let $S'\subset G^{\mathbf
    m}$ be the union of sections for the action of $\HH_{\tilde
    c}$ on $p^{-1}(\tilde c)$, as $\tilde c$ varies in $\tilde
  S$.  Then $S'$ is a section for the action of $\HM$ on $G^\mm$.
\end{prop}
\begin{proof}
  We need to show that $\HM\cdot S'=G^\mm$ and that two elements $X,Y
  \in S'$ that belong to the same $\HM$-orbit coincide.

  To see that $\HM\cdot S'=G^\mm$, pick $X=(g_1, \lds, g_r)$ in
  $G^\mm$. Up to the action of $\Aut G$, we can assume that
  $\pi(X)\in S$. Then $p(X)=\sigma\cd \tilde c$ for some
  $\sigma\in(\Sigma_r)_\mm$ and some $\tilde c\in \tilde S$. If
  $\sigma=\rho(\nu)$, where $\rho$ is the map in \eqref{rhouta}, then
  $\nu$ is in $(\zopfr)_{\mm}$ and
  $\nu\meno \cd X =\nu^{-1}\cdot(g_1,\dotsc, g_r)$ is in
  $p^{-1}(\tilde c)$, so its $\HH_{\tilde c}$-orbit intersects
  $S'$. As $\HH_{\tilde c} \subset \HM$ we conclude that the
  $\HM$-orbit of $X$ intersects $S'$, as desired.

  Next assume that $X$, $Y$ are elements of $S'$ that belong to the
  same $\HM$-orbit. By equivariance $\pi(X)$ and $\pi(Y)$ are in the
  same $\HM$-orbit.  But $(\zopfr)_\mm$ acts trivially on
  $\mathcal{C}_G^\mm/(\Sigma_r)_\mm$.  Hence $\pi(X)$ and $\pi(Y)$ are
  in the same $\Aut G$-orbit. By construction they also belong to $S$;
  therefore, they coincide. Since $\tilde c=p(X)$ and $p(Y)$ are
  elements of $\tilde S$ that lie over the same element of
  $\mathcal{C}_G^{\mathbf{m}}/{(\Sigma_r)_\mathbf{m}}$, they also
  coincide. By equivariance, $X$ and $Y$ are in the same $\HM$-orbit
  only if they are in the same $\HH_{\tilde c}$-orbit, which forces
  them to be the same.
\end{proof}
  
A first advantage of this approach is that for some $G$ and
$\mathbf{m}$, $p^{-1}(\tilde c)$ can be considerably smaller than the
whole $G^{\mathbf{m}}$, leading to a reduced memory usage. In
addition, $\HH_{\tilde c}$ has index
\[[\HM:\HH_{\tilde c}]=[(\braid)_{\mathbf m} : (\braid)_{\tilde c}]
  \cd [\Aut G : (\Aut G)_{\tilde c}].\] This means that
$\HH_{\tilde{c}}$ is typically smaller than $\HM$ and the image of
$\HH_{\tilde c}$ in $\operatorname{Perm}(p^{-1}(\tilde c))$ is smaller
than the image of $\HM$ in $\operatorname{Perm}(G^{\mathbf{m}})$, in a
way that more than compensates for the fact that one has to iterate
through elements of $S$.

\begin{ex}
  \label{example:Sigma4}
  For example, take $G=\Sigma_4$ and
  $\mathbf m=(\underbrace{2,\dots,2}_r)$. Then $G$ has nine elements
  of order two, namely the conjugacy class of $(1,2)$ and the
  conjugacy class of $(1,2)(3,4)$, call them $C_1$ and $C_2$. Then
  $|C_1|=6$, $|C_2|=3$ and
  \[G^{\mathbf m}=(C_1\cup C_2)^r\] has $9^r$ elements.

  However, if
  \begin{equation}
    \label{eqn:tildecsigma4}
    \tilde c=(\underbrace{C_1,\dotsc, C_1}_k,\underbrace{C_2,\dotsc, C_2}_{r-k}),
  \end{equation}
  up to the order, then $p^{-1}(\tilde c)$ has $6^k 3^{r-k}$ elements,
  which is considerably less.  It follows from \eqref{pVppa} that
  $(\zopfr)_{\tilde c} = \rho\meno ((\Sigma_r)_ {\tilde c}) $, so
  $\zopfr / (\zopfr)_{\tilde c} \cong \Sigma_r /(\Sigma_r)_{\tilde
    c}$. By \eqref{eqn:tildecsigma4}
  $(\Sigma_r)_{\tilde {c}} \cong \Sigma_k \times \Sigma_{r-k}$, so
  $[\zopfr : (\zopfr)_{\tilde c} ] = \binom{r}{k}$.  Moreover in this
  special case $\Aut G$ fixes every conjugacy class, as every
  automorphism of $G=\Sigma_4$ is inner \cite[vol. 1 Satz 5.5
  p. 175]{Huppert} or \cite[Cor. 7.7 p. 159]{rotman}.  Hence $\Aut G$
  fixes $\tilde c$, and
  $\HH_{\tilde c}=(\braid)_{\tilde c}\times \Aut G$ has index
  $\binom{r}{k}$ in $\HM$. Furthermore,
  $\mathcal{C}_G^{\mathbf{m}} / (\Sigma_r)_{\mathbf{m}} $ has $r+1$
  elements indexed by $k$ as above and $\Aut G$ acts trivially. We
  have split the action into $r+1$ actions of much smaller size.
\end{ex}

\begin{say}
  In order to apply Proposition~\ref{prop:section}, one needs to
  compute the stabilizer $\HH_{\tilde c}$ of an element $\tilde c$ in
  $\mathcal{C}_G^{\mathbf{m}}$.
  
  In the above example $\HH_{\tilde{c}}$ splits as a product of
  $(\zopfr)_{\tilde{c}}$ and $(\Aut G)_{\tilde{c}}$. In general
  $(\zopfr)_{\tilde{c}} \times (\Aut G)_{\tilde{c}}$ is only a subgroup
  of $\HH_{\tilde{c}}$.  In fact an element of $\zopfr$ can move the
  conjugacy classes and element of $\Aut G$ can restore them to their
  original order.  The situation in the general case is described by
  the exact sequence
  \begin{equation}
    \label{eqn:Hc}
    1\to (\braid)_{\tilde c}\stackrel{\alfa}{\lra} \HH_{\tilde c}\stackrel{\beta}{\lra} (\Aut G)_c\to 1,
  \end{equation}
  where $\alfa$ is the inclusion: $\alf(\psi) = (\psi,1)$, $\beta$ is
  the projection $\beta(\psi, \eta) = \eta$, and $c$ is the image of
  $\tilde c$ in the quotient
  $\quozient{\mathcal{C}_G^{\mathbf{m}}}{(\Sigma_r)_\mm}$.  Indeed if
  $(\psi,\eta) \in \HH_{\tilde{c}}$ and $\rho(\psi) = \sigma$, then by
  \eqref{pVppa} we have
  \begin{gather*}
    (\psi,\eta)\cd (\tilde{c}_1, \lds, \tilde{c}_r) =
    (\eta(\tilde{c}_{\sigma(1)}), \lds, \eta(\tilde{c}_{\sigma(r)}) )
    = (\tilde{c}_1, \lds, \tilde{c}_r).
  \end{gather*}
  Thus $\eta$ permutes the elements $\tilde{c}_i$ i.e. it fixes $c$.
  This shows that $\beta$ lands in $(\aug)_c$.  Obviously
  $\beta \alfa = 1$.  If $(\psi,\eta) \in \ker \beta$, then
  $\eta=\id_G$ hence $\sigma =1$ and $\psi \in (\zopfr)_{\tilde{c}}$.
  If $\eta \in (\aug)_c$, then $\eta(c_i) = c_{\sigma(i)}$. Since
  $\rho$ is surjective there is $\psi \in \zopfr$ such that
  $\rho(\psi) = \sigma^{-1}$ and then
  $(\psi,\eta) \in \HH_{\tilde{c}}$. Thus $\beta$ is onto.

  In particular, one can obtain a set of generators for
  $\HH_{\tilde c}$ by choosing elements $\gamma_1,\dotsc, \gamma_k$
  such that $\beta(\gamma_1),\dotsc, \beta(\gamma_k)$ generate
  $(\Aut G)_c$ and adding a set of generators of
  $(\braid)_{\tilde c}$.

  So we need to find a set of generators of $(\braid)_{\tilde c}$.  We
  start with the following observation.  An element of
  $(\mathcal{C}_G)^r/\Sigma_r$ (or of
  $\mathcal{C}_G^\mm / \Sigma_\mm$) is represented in our
  implementation as a \emph{multiset}, i.e.
  $(C_1^{r_1}, \lds, C_s^{r_s})$ where $C_i \in \mathcal{C}_G$ and
  $r_i \geq 1$.  To each multiset of conjugacy classes $c$ there
  correspond several ordered sequences of conjugacy classes
  $\tilde{c}$.  We fix a total ordering on the set of conjugacy
  classes $\mathcal{C}_G$.  Then, among all sequences $\tilde{c}$
  corresponding to the same $c$ there is a minimal one with respect to
  the induced lexicographic order on $\mathcal{C}^r_G$, i.e.  the only
  nondecreasing one:
  \begin{gather}
    \label{pio}
    \tilde{c}=(C_1,\dotsc, C_r) \quad \text{ with }\quad C_1\leq \dots
    \leq C_r.
  \end{gather}
  For ease of notation, multisets will be represented by the
  associated nondecreasing sequence in the following.
  
  In order to find generators of $(\braid)_{\tilde c}$ one can
  consider the exact sequence
  \[0\to \mathcal{P}\braid \to \braid \stackrel{\rho}{\to} \Sigma_r\to
    0,\] where $\mathcal{P}\braid$ denote the group of pure braids,
  i.e. $\ker \rho$.  We also have the following exact sequence
  obtained by restriction:
  \[0\to \mathcal{P}\braid \to (\braid)_{\tilde c} \to
    (\Sigma_r)_{\tilde c}\to 0.\] Recall from \cite[p. 20]{birman}
  that $\mathcal{P}\braid$ is generated by the elements
  \[A_{ij}=\sigma_{j-1}\sigma_{j-2}\dotsm\sigma_{i+1}\sigma_i^2\sigma_{i+1}^{-1}\dotsm
    \sigma_{j-2}^{-1}\sigma_{j-1}^{-1}, \quad 1\leq i<j\leq r.\] Since
  we chose $\tilde{c}$ with $C_1\leq \dots \leq C_r$, the subgroup
  $(\Sigma_r)_{\tilde c}$ is generated by transpositions $(i\ i+1)$
  where $i$ is such that $C_i=C_{i+1}$.  It follows that
  $(\braid)_{\tilde c}$ is generated by
  \[\{A_{ij}\mid 1\leq i<j\leq r\}\cup \{\sigma_i \mid 1\leq i\leq
    r-1, C_i=C_{i+1}\};\] notice however that when $C_i=C_j$, then
  $C_i=C_{i+1}=\dots = C_j$, so $A_{ij}$ belongs to the group
  generated by $\sigma_i,\dotsc, \sigma_{j-1}$, and it is redundant as
  a generator of $(\braid)_{\tilde c}$. This gives
  Algorithm~\ref{algo:Hc}.

  \begin{algorithm}[th] \SetKwInOut{Input}{input}
    \SetKwInOut{Output}{output} \Input{A group $G$ and an element
      $\tilde c=(C_1,\dotsc, C_r)\in\mathcal{C}_G^r$,
      $C_1\leq \dots\leq C_r$} \Output{A set of generators for
      $\HH_{\tilde c}$}
    $\Gamma\leftarrow \{A_{ij}\mid 1\leq i<j\leq r, C_i\neq C_j\}\cup
    \{\sigma_i \mid 1\leq i\leq r-1, C_i=C_{i+1}\}$\; \For{$\phi$ in a
      set of generators for $(\Aut G)_{c}$}{ $\sigma\leftarrow$ a
      permutation such that $ \phi\cd \tilde c=\sigma\cd\tilde c$\;
      $\alpha_{i_1}\dotsm \alpha_{i_k}\leftarrow$ a decomposition of
      $\sigma$ as a product of transpositions $\alpha_j=(j\,j+1)$\;
      $\psi\leftarrow\sigma_{i_1}\dotsm \sigma_{i_k}$\; add
      $(\psi^{-1},\phi)\in\braid\times\Aut G$ to $\Gamma$ } \Return
    $\Gamma$
    \caption{Computing $\HH_{\tilde c}$}\label{algo:Hc}
  \end{algorithm}
\end{say}

  \begin{say}
    The problem considered in this paper is to compute effectively a
    section for the action of $\HH$ on $\datar(G)$. Logically, the
    problem can be split in two parts: first computing a section for
    the action of $\HH$ on $G^r$, next checking which elements of the
    section are spherical systems of generators of $G$. For reasons of
    efficiency, our algorithm does not attack the two problems one
    after another, but simultaneously.
    
    Proposition~\ref{prop:section} reduces the first problem to
    determining a section for the action of $\HH_{\tilde c}$ on each
    $p^{-1}(\tilde c)$. In view of the second part, two more
    optimizations are important, already used by Breuer \cite{breuer}
    and Paulhus \cite{paulhus}.  Indeed, for some elements $c$, one
    can ascertain \emph{a priori} that $\pi^{-1}(c)=p^{-1}(\tilde c)$
    does not contain any system of generators at all!
  \end{say}

  This is based on a theorem of Frobenius, see
  \cite[p. 406]{lando-zvonkin} for a proof and also \cite{jones} for a
  generalization to higher genus.
  \begin{teo}[Frobenius's formula] Given a finite group $G$ and
    conjugacy classes $C_1,\dotsc, C_r$, the number of $r$-ples
    $(g_1,\dotsc, g_r)\in C_1\times \dots \times C_r$ such that
    $\prod g_i=1$ is
    \[\frac{|C_1|\dotsm |C_r|}{|G|} \sum_{\chi} \frac{\chi(C_1)\dotsm
        \chi(C_r)}{\chi(1)^{r-2}},\] where the sum is over characters
    of irreducible representations of $G$.
  \end{teo} Thus, $p^{-1}(C_1,\dotsc, C_r)$ can only contain a system
  of spherical generators if
  $\sum_{\chi} \frac{\chi(C_1)\dotsm \chi(C_r)}{\chi(1)^{r-2}}$ is
  nonzero; in this case, we will say that $(C_1,\dotsc, C_r)$ passes
  Frobenius' test. Notice that this condition is independent of the
  order of the conjugacy classes $C_1, \lds, C_r$.

  \begin{ex} In the setting of Example~\ref{example:Sigma4} it is easy
    to check, by looking at the character table of $\Sigma_4$, that
    Frobenius formula evaluates to zero for $k$ odd in
    \eqref{eqn:tildecsigma4}. This eliminates half of the elements of
    $S$. Notice that in this particular case, the same conclusion can be reached by observing that when elements $g_1,\dotsc, g_r$ of $\Sigma_4$  satisfy $\prod g_i=1$, then the product of their signs must be $1$.
  \end{ex}

  A second condition is based on a theorem of Scott \cite[Theorem
  1]{Scott}:
  \begin{teo}[Scott]
    \label{thm:Scott} Let $G$ be a group generated by
    $g_1,\dotsc, g_n$ with $g_1\dotsm g_n=1$ and let $V$ be a
    finite-dimensional representation of $G$ over any field. Then
    \[\sum_{i=1}^n v(g_i)\geq v(G)+v(G^*),\] where $v(g_i)=\codim
    V^{g_i}$, $v(G)=\codim V^G$, $v(G^*)=\codim (V^*)^G$.
  \end{teo} Since $v(g_i)$ only depends on the conjugacy class of
  $g_i$, and the order of the $g_i$ is irrelevant for the condition,
  we see that Scott's theorem determines a test to identify the
  $c\in \mathcal{C}_G^{\mathbf{m}}/{(\Sigma_r)_\mathbf{m}}$ which can
  potentially have a system of generators in their preimage. We will
  say that $c$ passes Scott's test over $\K$ if the condition of
  Theorem~\ref{thm:Scott} is satisfied for every finite-dimensional
  irreducible representation of $G$ over $\K$.

  For any field $\K$, we define
  \[\mathcal{C}_G^{\mm,\K}=\{\tilde c\in\mathcal{C}_G^\mm\mid \tilde c
    \text{ passes Frobenius' test and Scott's test over $\K$}\}.\] Our
  implementation runs the test on a field $\F_q$, with $q$ a fixed
  prime number greater than the group order. Since
  $\mathcal{C}_G^{\mm,\F_q}$ is invariant under
  $(\Sigma_r)_{\mathrm{m}}$ and $\Aut G$, this allows us to replace
  $\mathcal{C}_G^{\mm}$ with $\mathcal{C}_G^{\mm,\F_q}$ in \S
  \ref{say:determinesection}.  \medskip

  \begin{ex} In Example~\ref{example:Sigma4}, consider the
    two-dimensional representation of $\Sigma_4$ obtained by pulling
    back the two-dimensional irreducible representation of $\Sigma_3$
    (the morphism $\Sigma_4\to\Sigma_3$ being given by the quotient by
    the group generated by $C_2$). Then
    \[v(C_1)=1, v(C_2)=0, v(\Sigma_4)=2=v(\Sigma_4^*).\] Thus, Scott's
    test eliminates all $\tilde c$ of the form
    \eqref{eqn:tildecsigma4} such that $k < 4$.
  \end{ex}

  \begin{say} We also employed a technical optimization which deserves
    to be mentioned.

    The condition $\langle g_1,\dotsc, g_{r-1}\rangle=G$ only depends
    on the set $\{g_1,\dotsc, g_{r-1}\}$, call it the \emph{underlying
      set} of the $r$-ple $(g_1,\dotsc, g_r)$. Different elements in
    $G^{\mathbf m}$ can have the same underlying set; indeed, when
    $c,c'\in \mathcal{C}_G^{\mathbf{m}}/{(\Sigma_r)_\mathbf{m}}$
    contain the same conjugacy classes, possibly with different
    multiplicities, elements in $\pi^{-1}(c)$ and elements in
    $\pi^{-1}(c')$ can have the same underlying set.

    Thus, when iterating over the sets $p^{-1}(\tilde c)$ of potential
    spherical systems of generators, it makes sense to store in memory
    a list of underlying sets that are known to either generate $G$ or
    not, and look up each underlying set in the list before actually
    performing the (costly) test to see whether a given $r$-ple
    actually generates the group. Since these lists can become quite
    large, one can save memory by observing that an element of
    $C_1\times\dots \times C_r$ cannot have the same underlying set as
    an element of $C'_1\times\dots \times C'_r$ unless
    $\{C_1,\dotsc, C_{r-1}\}=\{C'_1,\dotsc, C'_{r-1}\}$.  Therefore,
    given a section $F$ of $\mathcal{C}_G^{\mm,\F_q}$, we work
    separately on each component $F_A$ of the partition
    \begin{equation*}
      \label{eqn:Fpartitioned} F=\bigsqcup_{A\subset \mathcal{C}_G} F_A,
      \quad F_A=\{(C_1,\dotsc, C_r)\in F\mid \{C_1,\dotsc, C_{r-1}\}=A\}.
    \end{equation*}    
    The partition is computed with Algorithm~\ref{algo:section}.
  \end{say}

  \begin{algorithm}[H] \SetKwInOut{Input}{input}
    \SetKwInOut{Output}{output} \SetKwComment{tcp}{/*}{*/} \Input{A
      group $G$, a signature $\mm$, a prime $q\geq \abs{G}$} \Output{A
      section $F$ of $\mathcal{C}_G^{\mm,\F_q}$, partitioned as in
      \eqref{eqn:Fpartitioned}} $F_A\leftarrow\emptyset$ for all
    $A\subset\mathcal{C}_G$\;
    $K\leftarrow \{(C_1,\dotsc, C_r)\in\mathcal{C}_G^r\mid
    \mbox{$C_1\leq\dots\leq C_r$} \text{ and some permutation }
    (C_{\sigma_1},\dotsc, C_{\sigma_r}) \text{ is in }
    \mathcal{C}^\mm_G\}$\;
    \tcp{$K=\mathcal{C}_G^\mm / (\Sigma_r)_\mm$, represented as in
      \eqref{pio})} $S\leftarrow$
    a section for the action of $\Aut G$ on $K$\\
    \For{$(C_1,\dotsc, C_r)$ in $S$} {\If{$(C_1,\dotsc, C_r)$ passes
        Frobenius' test and Scott's test over $\F_q$} {add the
        sequence $(C_1,\dotsc, C_r)$ to the set
        $F_{\{C_1,\dotsc, C_{r-1}\}}$} } \Return $F=\bigcup F_A$
    \caption{Computing the section}\label{algo:section}
  \end{algorithm}

  \begin{say} In theory, a section in
    $\xgm := \datar(G)\cap G^\mm \subset G^\mm $ can be obtained by
    computing $S'$ as in Proposition~\ref{prop:section}, then
    verifying for each element $(g_1,\dotsc, g_r)$ whether it is a
    spherical system of generators. A bit of experimenting shows that
    it is better to identify the subset of $p^{-1}(\tilde c)$
    consisting of generators, before applying the action of
    $\HH_{\tilde c}$ to extract a section.

    We also point out that in a spherical system of generators
    $(g_1,\dotsc, g_r)$, the last element is determined by the others;
    thus, given
    $\tilde c=(C_1,\dotsc, C_r)\in \mathcal{C}_G^{\mathbf m}$,
    spherical generators in $p^{-1}(\tilde c)$ can be determined by
    iterating in $C_1\times \dots \times C_{r-1}$, and testing for
    each element if the inverse of the product is in $C_r$.
  
    The number of iterations can therefore be reduced by choosing
    $\tilde c=(C_1,\dotsc, C_r)$ in such a way that the last conjugacy
    class is the biggest. Since we have $C_1\leq\dots\leq C_r$
    relative to the fixed ordering of $\mathcal{C}_G$, it suffices to
    choose the latter in such a way that conjugacy classes with more
    elements come after.
  
    This leads to Algorithm~\ref{algo:fixedGm}.

    \begin{algorithm}[th] \SetKwInOut{Input}{input}
      \SetKwInOut{Output}{output} \SetKw{Return}{return}
      \SetKwComment{tcp}{//}{} \Input{A group $G$ of order $d$; a
        signature $\mm\in\mathfrak{S}_{d,g_{\max}}$} \Output{One
        representative in each Hurwitz equivalence class of spherical
        systems of generators of $G$ with signature $\mm$}
      \caption{Classifying spherical systems of generators for fixed
        $G$, $\mm$ \label{algo:fixedGm}} $q\leftarrow$ the smallest
      prime number greater than $d$\; $F\leftarrow$ {a section of
        $\mathcal{C}_G^{\mm,\F_q}$, partitioned as in
        \eqref{eqn:Fpartitioned}}\; \For{$A\subset\mathcal{C}_G$}{
        generating $\leftarrow \{\}$\; notgenerating
        $\leftarrow \{\}$\;{ \For{$(C_1,\dotsc, C_{r})$ in $F_A$}{
            $X\leftarrow\{\}$\; \For{$(g_1,\dotsc, g_{r-1})$ in
              $C_1\times\dotsm \times C_{r-1}$}{
              $g_r\leftarrow (g_1\dotsm g_{r-1})^{-1}$\;
              \If{$g_r\in C_r$ and $\{g_1,\dotsc, g_{r-1}\}\notin$
                notgenerating}{ \eIf{$\{g_1,\dotsc, g_{r-1}\}\in$
                  generating or
                  $\langle g_1,\dotsc, g_{r-1}\rangle =G$}{ add
                  $\{g_1,\dotsc, g_{r-1}\}$ to generating\; add
                  $(g_1,\dotsc, g_r)$ to $X$\; }{ add
                  $\{g_1,\dotsc, g_{r-1}\}$ to notgenerating\; } } }
            \If{$X$ not empty}{ $\HH_{\tilde c}\leftarrow $ stabilizer
              of $(C_1,\dotsc, C_r)$ in $\HM$\; append to output a
              section of $X$ for the action of $\HH_{\tilde c}$ } } }
      }
    \end{algorithm}

  \end{say}

  \begin{say}
    \label{say:dihedral}
    For dihedral groups of order $4k+2$, our problem is the object of
    one of the main results in \cite{baffo-linceo}: Theorem 2 of that
    paper shows that given $\tilde c$ in $\mathcal{C}_G^{\mm}$ there
    is at most one orbit of systems of spherical generators in
    $p^{-1}(\tilde c)$. Indeed, if the order is of the form $4k+2$,
    the {\it numerical type} defined in \cite[Definition
    2]{baffo-linceo} corresponds exactly to the class of $c$ modulo
    the action of $\Aut G$.
  
    Therefore, instead of computing the whole set of systems of
    generators mapping to $\tilde c$ as in
    Algorithm~\ref{algo:fixedGm}, it is sufficient to iterate through
    $p^{-1}(\tilde c)$ and stop as soon as a system of generators is
    found.
  \end{say}

  \begin{say} For abelian groups $G$, conjugacy classes contain a
    single element and the map $p$ of Diagram~\ref{diag:Gr} is
    injective, with the action of $\braid$ reducing to an action of
    $\Sigma_r$. Therefore, having computed a section $S$ for the
    action of $\Aut G$ on $\mathcal{C}_G^{\mm}/(\Sigma_r)_\mm$ exactly
    as in the nonabelian case, one only needs to determine for each
    element of $S$ whether its preimage in $G^{\mm}$ is a spherical
    system of generators. The tests of Scott and Frobenius become
    redundant here, since for fixed elements $(g_1,\dotsc, g_r)$ it is
    more efficient to check the conditions $\prod g_i=1$ and
    $\langle g_1,\dotsc, g_r\rangle =G$ directly.
  \end{say}

  \begin{say}
    \label{say:results}
    Running our implementation \cite{gullinbursti} of the
    above-illustrated algorithms, we have been able to classify
    topological types of holomorphic actions on Riemann surfaces
    (equivalently of orientation-preserving actions on orientable
    topological surfaces) of genus $g\leq 39$ with genus 0 quotient,
    with only three exceptions:
    \begin{align*}
      G&=(\Z_3\times\Z_3)\rtimes\Z_2,  & {\mm}&= \{ 2, 2, 2, 2, 2, 2, 2, 2, 2, 2 \}, & g&=28;\\
      G&=(\Z_3\times\Z_3)\rtimes\Z_2,   & {\mm}&= \{ 2, 2, 2, 2, 2, 2, 2, 2, 2, 2, 3 \}, & g&=34;\\
      G&=(\Z_3\times\Z_3)\rtimes\Z_2,   & {\mm}&= \{  2, 2, 2, 2, 2, 2, 2, 2, 2, 2, 2, 2 \}, &  g&=37.
    \end{align*}
    In these cases, the number of spherical systems of generators is
    too large to fit into the memory of the computer at our disposal.

    The number of topological types by genus is summarized in
    Table~\ref{table:numberofssgbyg}. A strict inequality such as
    $>3580$ for $g=28$ refers to the fact that the program classifies
    $3580$ topological types with groups and signatures distinct from
    the offending group and signature, which gives rise to at least
    one more topological type.

    The three exceptions above, for which we were not able to finish
    the computation, all share the same group,
    i.e. $G=(\Z_3\times\Z_3)\rtimes\Z_2$. This is the semidirect
    product of $A:= \Z_3 \times \Z_3$ with $\Z_2$ defined by the
    morphism $[1]_2 \mapsto \phi \in \Aut A$, where $\phi(a) = -a$; it
    is denoted by \texttt{Smallgroup(18,4)} in the \magma database.
    This group belongs to the family of so-called \emph{generalized
      dihedral groups}, i.e. groups of the form $G= A \rtimes \Z_2$,
    with $A$ abelian and morphism $[1]_2\mapsto \phi$ as above.  When
    the order of $A$ is odd, these groups are naturally challenging
    for our algorithm, since they have a very big conjugacy class, the
    complement of the index $2$ subgroup $A$. For $A$ cyclic, i.e. for
    $G$ a ``standard'' dihedral group, we were able to avoid heavy
    computations by applying the results of \cite{baffo-linceo}, as
    explained in \S \ref{say:dihedral}. We suspect that a more
    complicated analysis could yield a similar result also for more
    general groups of the form $A\rtimes \Z_2$. Apart from the
    theoretical importance, this would allow to reach the
    classification of topological types up to 39 or 40 without need of
    more computations. More precisely the invariant in \cite{clp}
    might give results analogous to those in \cite{baffo-linceo} for a
    wider class of generalized dihedral groups. We hope to be able to
    give some results in this direction in a forthcoming paper.

  \end{say}

  \begin{table}
    \caption{\label{table:numberofssgbyg} Number of topological types
      of Galois covers of the line with genus $g$}
    \[
      \begin{array}{cc}
        g & \text{\# types}\\
        \hline
        2 & 19 \\
        3 & 46 \\
        4 & 65 \\
        5 & 92 \\
        6 & 95 \\
        7 & 160 \\
        8 & 129 \\
        9 & 343 \\
        10 & 289 \\
        11 & 342 \\
        12 & 317 \\
        13 & 741 \\
        14 & 323 
      \end{array}
      \hspace{1cm}
      \begin{array}{cc}
        g & \text{\#  types}\\
        \hline

        15 & 768 \\
        16 & 687 \\
        17 & 1473 \\
        18 & 711 \\
        19 & 1689 \\
        20 & 881 \\
        21 & 2790 \\
        22 & 1546 \\
        23 & 2178 \\
        24 & 1852 \\
        25 & 5955   \\
        26 & 1881 \\ 
        27 & 4351 
      \end{array}
      \hspace{1cm}
      \begin{array}{cc}
        g & \text{\#  types}\\
        \hline
        28 & > 3580 \\
        29 & 8169 \\
        30 & 3992 \\
        31 & 8506 \\
        32 & 4336 \\
        33 & 16007\\
        34 & > 6983 \\
        35 & 11827 \\
        36 & 8753 \\
        37 & > 26712 \\
        38 & 8486 \\
        39 & 19099 \\
        \hspace{0cm}
      \end{array}
    \]
  \end{table}

\end{document}